\documentclass[reqno]{amsart}
\usepackage{amsmath,amssymb}
\usepackage{verbatim}
\usepackage{color}

\newtheorem{thm}{Theorem}[section]
\newtheorem{lemma}[thm]{Lemma}
\newtheorem{cor}[thm]{Corollary}

\theoremstyle{definition}
\newtheorem{definition}[thm]{Definition}
\newtheorem{ex}[thm]{Example}

\theoremstyle{remark}
\newtheorem{remark}[thm]{Remark}
\newtheorem{remarks}[thm]{Remarks}

\numberwithin{equation}{section}

\newenvironment{mylist}{\begin{enumerate}

}{\end{enumerate}}

\newcommand\al{\alpha}
\newcommand\be{\beta}
\newcommand\ga{\gamma}
\newcommand\Ga{\Gamma}
\newcommand\de{\delta}
\newcommand\De{\Delta}
\newcommand\ep{\epsilon}

\newcommand\la{\lambda}

\newcommand\om{\omega}

\newcommand\si{\sigma}

\renewcommand\th{\theta}

\newcommand\bal{\boldsymbol{\alpha}}
\newcommand\bbe{\boldsymbol{\beta}}
\newcommand\bfeta{\boldsymbol{\eta}}

\newcommand\bzero{{\boldsymbol{0}}}

\newcommand\tk{\widetilde k}
\newcommand\tu{\widetilde u}
\newcommand\ts{\widetilde s}
\newcommand\tw{\widetilde w}

\newcommand\tla{\widetilde \la}
\newcommand\tth{\widetilde \th}

\newcommand\tI{\widetilde I}
\newcommand\tX{\widetilde X}
\newcommand\tY{\widetilde Y}

\newcommand\tDe{\widetilde \De}

\newcommand\R{\mathbb{R}}

\newcommand\B{{\mathcal B}}

\newcommand\X{\times}

\newcommand{\pa}{\partial}

\let\le=\leqslant
\let\ge= \geqslant

\renewcommand\emptyset{\mbox{\Large \o}}

\newcommand\Cs{{C^1_{\rm s}}}

\begin{document}

\title[Multi-point problems with
Sturm-Liouville-type boundary conditions]
{Linear, second-order problems with
Sturm-Liouville-type multi-point boundary conditions}
\author{Bryan  P.  Rynne}
\address{Department of Mathematics and the Maxwell Institute for
Mathematical Sciences, Heriot-Watt University,
Edinburgh EH14 4AS, Scotland.}
\email{bryan@ma.hw.ac.uk}

\begin{abstract}
We consider the linear eigenvalue problem
consisting of the equation
\begin{equation} \tag{1}
 -u'' = \lambda u ,  \quad \text{on $(-1,1)$},
\end{equation}
where $\lambda \in \mathbb{R}$,
together with the general multi-point boundary conditions
\begin{equation} \tag{2}
\alpha_0^\pm u(\pm 1)  + \beta_0^\pm u'(\pm 1) =
\sum^{m^\pm}_{i=1} \alpha^\pm_i u(\eta^\pm_i)
 +
\sum_{i=1}^{m^\pm} \beta^\pm_i u'(\eta^\pm_i) ,
\end{equation}
where
$m^\pm \ge 1$ are integers,
$\alpha_0^\pm,\beta_0^\pm \in \mathbb{R}$,
and, for each
$i = 1,\dots,m^\pm$,
the numbers
$\alpha_i^\pm,\beta_i^\pm \in \mathbb{R}$,
and
$\eta_i^\pm \in [-1,1]$,
with $\eta_i^\pm \ne \pm 1$.
We  also suppose that:
\begin{gather}
\alpha_0^\pm \ge 0, \quad \alpha_0^\pm + |\beta_0^\pm| > 0 , \tag{3}
\\
\pm \beta_0^\pm \ge 0 , \tag{4}
\\
\left( \frac{\sum_{i=1}^{m^\pm} |\alpha_i^\pm|}{\alpha_0^\pm} \right)^2
 +
\left( \frac{\sum_{i=1}^{m^\pm} |\beta_i^\pm|}{\beta_0^\pm} \right)^2
 < 1 , \tag{5}
\end{gather}
with the convention that if any denominator in
(5) is zero then the corresponding numerator must
also be zero, and the corresponding fraction is omitted from (5)
(by (3), at least one denominator is nonzero in each condition).

An {\em eigenvalue} is a number $\lambda$ for which (1)-(2),
has a non-trivial solution $u$ (an {\em eigenfunction}),
and the {\em spectrum}, $\si$,  is the set of eigenvalues.
In this paper we show that the basic spectral
properties of this problem are similar to those of the
standard Sturm-Liouville problem with separated boundary conditions.
Similar multi-point problems have been considered before
under more restrictive hypotheses.
For instance, the cases where
$\beta_i^\pm = 0$, or $\alpha_i^\pm = 0$, $i = 0, \dots, m^\pm$
(such conditions have been termed Dirichlet-type or Neumann-type
respectively),
or the case of a single-point condition at one end point and a
Dirichlet-type or Neumann-type multi-point condition at the other end.
Different oscillation counting methods have been used in each of these
cases, and the results here unify and extend all these previous results
to the above general Sturm-Liouville-type boundary conditions.
\end{abstract}

\maketitle

\section{Introduction}  \label{intro.sec}

%
We consider the linear eigenvalue problem
consisting of the equation
\begin{equation} \label{eval_de.eq}
 -u'' = \la u ,  \quad \text{on $(-1,1)$},
\end{equation}
where $\la \in \R$,
together with the general multi-point boundary conditions
\begin{equation} \label{slbc.eq}
\al_0^\pm u(\pm 1)  + \be_0^\pm u'(\pm 1) =
\sum^{m^\pm}_{i=1} \al^\pm_i u(\eta^\pm_i)
 +
\sum_{i=1}^{m^\pm} \be^\pm_i u'(\eta^\pm_i) ,
\end{equation}
where
$m^\pm \ge 1$ are integers,
$\al_0^\pm,\be_0^\pm \in \R$,
and, for each
$i = 1,\dots,m^\pm$,
the numbers
$\al_i^\pm,\be_i^\pm \in \R$,
and
$\eta_i^\pm \in [-1,1]$,
with $\eta_i^\pm \ne \pm 1$.
We  write
$\al^\pm := (\al_1^\pm,\dots,\al_{m^\pm}^\pm) \in \R^{m^\pm}$,
and similarly for $\be^\pm$, $\eta^\pm$.
The notation $\al^\pm = 0$ or $\be^\pm = 0,$ will mean the zero vector
in $ \R^{m^\pm }$, as appropriate.
Naturally, an {\em eigenvalue} is a number $\la$ for which
\eqref{eval_de.eq}-\eqref{slbc.eq},
has a non-trivial solution $u$ (an {\em eigenfunction}).
The {\em spectrum}, $\si$,  is the set of eigenvalues.
Although the boundary conditions \eqref{slbc.eq} are non-local, for ease
of discussion we will usually say that the condition with
superscript $\pm$ holds
`at the end point $\pm 1$'.

Throughout we will suppose that the following conditions hold:
\begin{gather}
\al_0^\pm \ge 0, \quad \al_0^\pm + |\be_0^\pm| > 0 ,
\label{albe_nz.eq}
\\
\pm \be_0^\pm \ge 0 ,
\label{albe_sign.eq}
\\
\left( \frac{\sum_{i=1}^{m^\pm} |\al_i^\pm|}{\al_0^\pm} \right)^2
 +
\left( \frac{\sum_{i=1}^{m^\pm} |\be_i^\pm|}{\be_0^\pm} \right)^2
 < 1 ,
\label{AB_lin_cond.eq}
\end{gather}
with the convention that if any denominator in
\eqref{AB_lin_cond.eq} is zero then the corresponding numerator must
also be zero, and the corresponding fraction is omitted from
\eqref{AB_lin_cond.eq}
(by \eqref{albe_nz.eq}, at least one denominator is
nonzero in each condition).
The condition \eqref{albe_nz.eq} simply ensures that the boundary
conditions at $\pm 1$ actually involve the values
$u(\pm 1)$ or $u'(\pm 1)$.
We will describe the motivation and consequences of
\eqref{albe_sign.eq} and \eqref{AB_lin_cond.eq}
further here, and also in the following sections.

When $\al^\pm = \be^\pm = 0$ the multi-point boundary conditions
\eqref{slbc.eq} reduce to standard ({\em single-point}) separated
conditions at $x = \pm 1$,
and the overall multi-point problem
\eqref{eval_de.eq}-\eqref{slbc.eq}
reduces to
a separated, linear Sturm-Liouville problem.
Thus, we will term the conditions  \eqref{slbc.eq}
{\em Sturm-Liouville-type} boundary conditions.
The spectral properties of the separated problem are of course well
known,
see for example \cite{CL},
but the spectral properties of the above general multi-point problem
have not previously been obtained.
Indeed, it is only recently that the basic spectral properties of
any multi-point problems have been obtained, and these were obtained
under more restrictive assumptions on the boundary conditions.

Boundary value problems with multi-point boundary conditions have
been extensively studied recently,
see for example,
\cite{BF,DR,GR,GGM,GUP,LIU,MOR1,RYN3,RYN5,RYN6,WI,WL,XU},
and the references therein.
Many of these papers consider the problem on the interval $(0,1)$,
and impose a single-point Dirichlet or Neumann condition at the
end-point $x=0$,
and a multi-point condition at $x=1$.
In our notation, these particular single-point conditions
correspond to the special cases
$\be_0^- = 0$ or $\al_0^- = 0$, respectively
(as well as $\al^- = \be^- = 0$),
so of course are covered by our results here.
We have used the interval $(-1,1)$ in order to simplify the notation
for problems with multi-point boundary conditions at both end-points ---
our results are, of course, independent of the interval on which
the problem is posed.
Problems with a single-point boundary condition at one end-point can
often be treated using shooting methods
(starting at the end with the single-point condition)
and so are considerably simpler to deal with than problems having
multi-point boundary conditions at both end-points
(for which shooting is not possible).
Problems with multi-point conditions at both end-points have been
considered in \cite{GR,GUP,LIU,RYN5,RYN6}
(and in many references therein --- the bibliography in \cite{LIU}
is particularly extensive).

The papers \cite{RYN5} and \cite{RYN6} discussed the following
particular special cases, or {\em types}, of multi-point boundary
conditions:
\begin{alignat}{10}
\text{Dirichlet-type:}&  &\qquad&
\sum_{i=1}^{m^\pm} |\al_i^\pm| < 1 = \al_0^\pm,
&\quad&
\be_0^\pm = 0, \quad \be^\pm = 0 ;
\label{dir_type_bc.eq}
\\
\text{Neumann-type:}&  &&
\al_0^\pm = 0, \quad \al^\pm = 0,
&&
\sum_{i=1}^{m^\pm} |\be_i^\pm| < 1 = \be_0^\pm .
\label{neu_type_bc.eq}
\end{alignat}
This terminology is motivated by observing that a
Dirichlet-type (respectively Neumann-type)
condition reduces to a single-point
Dirichlet (respectively Neumann)
condition when
$\al=0$ (respectively $\be=0$).
The case of a  Dirichlet-type condition at one end point and a
Neumann-type
condition at the other end point was also discussed in \cite{RYN6},
where such conditions were termed {\em mixed}.
Clearly, the hypotheses
\eqref{dir_type_bc.eq} and \eqref{neu_type_bc.eq}
are special cases of the general hypothesis
\eqref{AB_lin_cond.eq},
and in these cases \eqref{albe_sign.eq} can be attained
simply by multiplying the boundary condition at $x=-1$ by $-1$,
so \eqref{albe_sign.eq}  is trivial.
Hence, our results here will unify and generalise all the results in
\cite{RYN5} and \cite{RYN6}.

It was shown in \cite{RYN5} and \cite{RYN6}  that the spectra of
these particular boundary values problems have many of the `standard'
properties of the spectrum of the separated Sturm-Liouville problem,
specifically:
\begin{mylist}
\item[($\si$-a)]
$\si$ is a strictly increasing sequence of real eigenvalues
$\la_k$, $k=0,1,\dots;$
\item[($\si$-b)]
$\lim_{k \to \infty} \la_k = \infty$;
\end{mylist}
for each $k \ge 0$:
\begin{mylist}
\item[($\si$-c)]
$\la_k$ has geometric multiplicity 1;
\item[($\si$-d)]
the eigenfunctions of $\la_k$ have
an `oscillation count' equal to $k$.
\end{mylist}
In the separated problem the oscillation count referred to in
property ($\si$-d) is simply the number of interior ({\em nodal}) zeros
of an eigenfunction.
However, in the multi-point problem it was found in \cite{RYN5} and
\cite{RYN6} that this method of counting eigenfunction oscillations no
longer yields property ($\si$-d), and alternative, slightly ad hoc,
methods were adopted, with different approaches being used for
different types of problem.
We will discuss this further below, and a more detailed discussion is
given in Section~9.4 of \cite{RYN6}.
Suffice it to say, for now, that the eigenfunction oscillation count we
adopt here, based on a Pr\"ufer angle approach
(see Section~\ref{Neu_nodal.sec}), extends and unifies the
disparate approaches adopted in \cite{RYN5} and \cite{RYN6}.

It was also shown in \cite{RYN5} and \cite{RYN6} that, in order to
obtain the spectral properties ($\si$-a)-($\si$-d), the conditions
\eqref{dir_type_bc.eq} and \eqref{neu_type_bc.eq}
are optimal for the
Dirichlet-type and Neumann-type conditions respectively,
in the sense that, in either of these cases, if the inequality $< 1$ in
\eqref{dir_type_bc.eq} or \eqref{neu_type_bc.eq}
is relaxed to $< 1 + \ep$, for any $\ep > 0$,
then $\si$ need not have have
all the properties ($\si$-a)-($\si$-d).
For the general Sturm-Liouville-type boundary conditions
\eqref{slbc.eq}
it will be shown here that if \eqref{albe_sign.eq} and
\eqref{AB_lin_cond.eq}
hold then $\si$  has the properties ($\si$-a)-($\si$-d),
and if either
\eqref{albe_sign.eq} or \eqref{AB_lin_cond.eq}
do not hold then $\si$ need not have all these properties.

\begin{remarks} \label{intro_remarks.rem}
(i)
Changing the length of the interval on which we consider the problem
rescales the coefficients $\be_0^\pm,\be^\pm$,
but not the coefficients $\al_0^\pm,\al^\pm$.
Such a change should not affect our hypotheses on the coefficients,
and indeed the condition
\eqref{AB_lin_cond.eq}
is invariant with respect to such a rescaling.
Thus, the form of condition \eqref{AB_lin_cond.eq} seems natural in
this respect.
\\[1 ex](ii)
In the separated case
(that is, when $\al^\pm = \be^\pm = 0$)
the sign condition \eqref{albe_sign.eq} ensures that $\la_0 > 0$
(except in the Neumann case, when $\la_0 = 0$),
and if this sign condition does not hold then negative eigenvalues may
exist.
It will be shown below that this is also true for the above
Sturm-Liouville-type boundary conditions
(assuming that \eqref{albe_nz.eq} and \eqref{AB_lin_cond.eq} hold);
it will also be shown that negative eigenvalues may have geometric
multiplicity 2.
Of course, this cannot happen in the separated problem
due to uniqueness of the solutions for initial value problems
associated with  \eqref{eval_de.eq}.
Hence, the full set of `standard' properties ($\si$-a)-($\si$-d)
need not hold if the sign condition \eqref{albe_sign.eq} is not
satisfied.
\\[1 ex](iii)
In principle, we should consider the possibility of complex
eigenvalues, especially as the problem is not `self-adjoint'
(without defining this precisely).
Indeed, if we did not impose the condition \eqref{AB_lin_cond.eq} then
complex eigenvalues could in fact occur.
However, with this condition it can be shown that all eigenvalues must
be real --- the proof is very similar to the proof of
Lemma~\ref{la_ge_zero.lem} below, which shows that under our hypotheses
the eigenvalues are positive.
In the light of this we will simply take it for granted throughout
the paper that all our coefficients, functions and function spaces are
real.
\\[1 ex](iv)
We primarily consider the spectral properties
($\si$-a)-($\si$-d) because of their potential applications to nonlinear
problems
(many of the cited references use eigenvalue properties to deal with
nonlinear problems, using relatively standard arguments such as
Rabinowitz' global bifurcation theory).
Of course, there are many other linear spectral properties that could be
investigated, such as eigenfunction expansions (the problem is
not self-adjoint, so this would not be trivial).
However, for brevity, we will omit any discussion of nonlinear problems
or other linear properties here.
\\[1 ex](v)
Boundary conditions having a more general non-local dependence on the
function $u$ than the finite sums of values at points in the interval
$(-1,1)$
(as in \eqref{slbc.eq})
have also been considered recently by several authors,
see for example \cite{WI} and the references therein.
These papers have considered Dirichlet-type and Neumann-type
boundary conditions in which the finite summations have been
replaced with Lebesgue-Stieltjes integrals,
see \cite{WI} for further details
(finite summations can be obtained simply by using step functions in
Lebesgue-Stieltjes integrals, so such integral conditions generalise
the finite summation conditions).
The methods and results below can readily be extended to deal with such
integral formulations of the boundary conditions --- the only
significant additional step required is dealing with the necessary
measure and integration theory.
These measure-theoretic details are described, for
Dirichlet-type and Neumann-type conditions, in \cite{GR}.
Since this step is relatively routine we will avoid all such
measure-theoretic difficulties here by simply considering the
finite summation conditions \eqref{slbc.eq}.
\end{remarks}

\subsection{Plan of the paper}  \label{plan.subsec}

The paper is organised as follows.
In Section~\ref{De_op.sec} we introduce various function spaces, and
then use these to define an operator realization of the multi-point
problem, and state the main properties of this operator.
In Section~\ref{single_bc.sec} we prove an existence and uniqueness
result for a problem consisting of equation \eqref{eval_de.eq} together
with a single, multi-point, boundary condition.
This problem could be regarded as a multi-point analogue of the usual
initial value problem for equation \eqref{eval_de.eq}.
We also give some counter examples which show that this uniqueness
result can fail in the multi-point setting when $\la < 0$.
As mentioned in Remark~\ref{intro_remarks.rem}-(ii),
the uniqueness result for this `multi-point, initial value
problem' then implies the simplicity of the eigenvalues of
\eqref{eval_de.eq},  \eqref{slbc.eq}, in the usual manner,
and the loss of this uniqueness can result in the existence
of eigenvalues having geometric multiplicity 2.
In particular, this shows the necessity of the sign condition
\eqref{albe_sign.eq}
if we wish to obtain all the properties ($\si$-a)-($\si$-d).

Our main results are obtained in Section~\ref{evals.sec}.
In Section~\ref{Neu_nodal.sec} we describe a Pr\"ufer angle method of
counting the oscillations of the eigenfunctions, and we then use this
technique in Section~\ref{spec.sec} to obtain our main results regarding
the properties of the spectrum.
We also show that this Pr\"ufer angle construction generalises and
unifies the various oscillation counting methods used in
\cite{RYN5} and \cite{RYN6}  in the Dirichlet-type, Neumann-type and
mixed cases respectively.
In Section~\ref{positivity.sec} we show that, under suitable
additional hypotheses, the principal eigenfunction is positive.
In Section~\ref{alg_mult.sec} we reinterpret the eigenvalues as the
characteristic values of the inverse operator constructed in
Section~\ref{De_op.sec}, and show that these characteristic values
have algebraic multiplicity 1;
this result then yields the value of the topological degree of an
associated linear operator.
In Section~\ref{eval_counterexamples.sec} we give some counter examples
to show the necessity of the hypothesis \eqref{AB_lin_cond.eq}.

\subsection{Some further notation}  \label{notation.subsec}

Clearly,  the eigenvalues $\la_k$ (and other objects to be introduced
below) depend on the values of the coefficients
$\al_0^\pm,\,\be_0^\pm,\,\al^\pm,\,\be^\pm,\,\eta^\pm,$
but in general we regard these coefficients as fixed,
and omit them from our notation.
However, at certain points of the discussion it will be convenient to
regard some, or all, of these coefficients as variable,
and to indicate the dependence of various functions on these
coefficients.
To do this concisely we will write:

$\bal_0 := (\al_0^-,\al_0^+) \in \R^2$
(for given numbers $\al_0^\pm \in \R$);

\ \,$\bal := (\al^-,\al^+) \in \R^{m^- + m^+}$
(for given coefficient vectors $\al^\pm \in \R^{m^\pm}$);
\\
and similarly for $\bbe_0,\,\bbe,\,\bfeta$.
We also define $\bzero := (0,0)  \in \R^{m^- + m^+}$.
We may then write, for example, $\la_k(\bal,\bbe)$
to indicate the dependence of $\la_k$ on $(\bal,\bbe)$.

In most of the paper we will regard $(\bal_0,\bbe_0)$ as
fixed, but at some points in the discussion it will be convenient to
allow $(\bal,\bbe)$
to vary, so long as the conditions
\eqref{albe_nz.eq}-\eqref{AB_lin_cond.eq}
continue to hold.
To describe this we define the following sets,
for any  $(\bal_0,\bbe_0) \in \R^4$ satisfying
\eqref{albe_nz.eq} and \eqref{albe_sign.eq}:
\begin{align*}
\B(\al_0^\pm,\be_0^\pm) &:=
\{ (\al^\pm,\be^\pm) \in \R^{2m^\pm} :
\text{$(\al_0^\pm,\be_0^\pm,\al^\pm,\be^\pm)$
satisfies \eqref{AB_lin_cond.eq}} \} ,
\\
\B(\bal_0,\bbe_0) &:=
\{ (\bal,\bbe) \in \R^{2(m^- + m^+)} :
\text{$(\bal_0,\bbe_0,\bal,\bbe)$
satisfies \eqref{AB_lin_cond.eq}} \}
\end{align*}
(so $\B(\bal_0,\bbe_0)$ is isomorphic to
$\B(\al_0^-,\be_0^-) \times \B(\al_0^+,\be_0^+)$);
we also define the set
$$
\B :=
\{ (\bal_0,\bbe_0,\bal,\bbe) \in \R^{2(2+m^- + m^+)} :
\text{\eqref{albe_nz.eq}-\eqref{AB_lin_cond.eq} hold} \} .
$$

At some points, when dealing with individual boundary conditions,
it will be convenient to let $\nu$ denote one of the signs $\{\pm\}$,
in which case, for a function $u$, the notation $u(\nu)$ will denote the
value of $u$ at the corresponding end point $\pm 1$.

\section{An operator realisation of the multi-point problem}
\label{De_op.sec}

For any integer $n \ge 0$, let $C^n[-1,1]$ denote the usual Banach
space of $n$-times continuously differentiable functions on $[-1,1]$,
with the usual sup-type norm, denoted by $|\cdot|_n$.
A suitable space in which to search for solutions of
\eqref{eval_de.eq},
incorporating the boundary conditions \eqref{slbc.eq},
is the space
\begin{align*}
X &:= \{u \in C^2[-1,1] :
\text{$u$ satisfies \eqref{slbc.eq}} \},
\\
\|u\|_X  &:= |u|_2  ,  \quad  u \in X.
\end{align*}
Letting
$Y:=C^0[-1,1]$, with the norm $\|\cdot \|_Y := |\cdot |_0$,
we now define  an operator $\De : X \to Y$ by
\[
\De u := u'',  \quad u \in X.
\]
By the definition of the spaces $X$, $Y$, the operator
$\De$ is a well-defined, bounded, linear operator,
and the eigenvalue problem
\eqref{eval_de.eq}-\eqref{slbc.eq}
can be rewritten in the form
$ -\De(u)  = \la u $, $u \in X $.
We will consider the eigenvalue problem in Section~\ref{spec.sec}
below, for now we will consider the invertibility of $\De$.

In the  Neumann-type case (that is, when $\al_0^\pm = 0$) it is clear
that any constant function $c$ lies in $X$, and $\De c = 0$,
so $\De$ cannot be invertible.
Thus, to obtain invertibility it is necessary to exclude
the Neumann-type case.
In view of the assumption \eqref{albe_nz.eq}, we can achieve this by
imposing the further condition
\begin{equation}  \label{al_pm_strict_pos.eq}
 \al_0^- +  \al_0^+ > 0.
\end{equation}
The following theorem shows that this condition is sufficient to ensure
invertibility of $\De$.

\begin{thm}  \label{De_inverse.thm}
Suppose that \eqref{albe_nz.eq}-\eqref{AB_lin_cond.eq} and
\eqref{al_pm_strict_pos.eq} hold.
Then $\De : X \to Y$ has a bounded inverse.
\end{thm}

\begin{proof}

We will show that the equation
\begin{equation}  \label{De_inverse.eq}
\De u = h, \quad  h \in Y,
\end{equation}
has a unique solution for all $h \in Y$.
Following the proof of Theorem~3.1 in \cite{RYN5}
(which considers Dirichlet-type conditions and constructs a solution of
\eqref{De_inverse.eq} via a compact integral operator)
shows that it suffices to prove the uniqueness of the solutions of
\eqref{De_inverse.eq}.
To prove this we observe that any solution $u_0$ of
\eqref{De_inverse.eq}
with $h=0$ must have the form
$u_0(x) = c_0 + c_1 x$, for some $(c_0,c_1) \in \R^2$,
and substituting $u_0$ into the boundary
conditions \eqref{slbc.eq} yields the pair of equations
\begin{equation} \label{De_inverse_coeffs.eq}
c_0 \Big( \al_0^\pm -  \sum_{i=1}^{m^\pm} \al_i^\pm \Big )
 +
c_1 \Big(\be_0^\pm - \sum_{i=1}^{m^\pm} \be_i^\pm
 \pm \al_0^\pm  - \sum_{i=1}^{m^\pm} \al_i^\pm \eta_i^\pm \Big)
= 0 .
\end{equation}
It now follows from
\eqref{albe_nz.eq}-\eqref{AB_lin_cond.eq}
that
$$
\al_0^\pm - \sum_{i=1}^{m^\pm} \al_i^\pm \ge 0,
\quad
\pm \Big(\be_0^\pm - \sum_{i=1}^{m^\pm} \be_i^\pm
 \pm \al_0^\pm  - \sum_{i=1}^{m^\pm} \al_i^\pm \eta_i^\pm \Big)
 > 0,
$$
and it follows from \eqref{al_pm_strict_pos.eq} that at least one of
the left hand inequalities here is strict.
These sign properties now ensure that the determinant associated with
the pair of equations \eqref{De_inverse_coeffs.eq} is non-zero,
so that $(c_0,c_1) = (0,0)$ is the unique solution of
\eqref{De_inverse_coeffs.eq}.
This proves the desired uniqueness result for \eqref{De_inverse.eq}, and
hence proves the theorem.
\end{proof}

In applications, continuity properties of the inverse operator
$\De^{-1}$ with respect to the various parameters in the problem
are important.
We will describe one such result ---
other such results could be obtained in a similar manner.

\begin{cor} \label{De_inverse_cts.cor)}
The operator
$\De(\bal_0,\bbe_0,\bal,\bbe)^{-1} : Y \to C^2[-1,1]$
depends continuously on
$(\bal_0,\bbe_0,\bal,\bbe) \in \B \setminus
\{ (\bal_0,\bbe_0,\bal,\bbe) : \al_0^- + \al_0^- = 0\}$
$($with respect to the usual topology for bounded linear operators$)$.
\end{cor}

\begin{proof}
The functions $\Phi^\pm$ in the construction of
$\De(\bal_0,\bbe_0,\bal,\bbe)^{-1}$ in the proof of
Theorem~\ref{De_inverse.thm} in \cite{RYN5} are continuous with respect
to $(\bal_0,\bbe_0,\bal,\bbe)$,
so the result follows immediately from that proof.
\end{proof}

\begin{remark}  \label{De_inverse.rem}
We have used the spaces $C^n[-1,1]$, $n=0,\,2,$
to define the operator $\De$,
and Theorem~\ref{De_inverse.thm} showed that the resulting
operator is invertible.
This is the function space setting that we will use here.
However, one could also use a Sobolev space setting to define
a similar operator as follows.
For arbitrary fixed $q \ge 1$,
let
\begin{align*}
\tY := L^q(-1,1),
\qquad
\tX &:= \{u \in W^{2,q}[-1,1] :
\text{$u$ satisfies \eqref{slbc.eq}}\}.
\end{align*}
Then  $\tDe : \tX \to \tY$ can be defined in the obvious manner,
and a similar proof to that of Theorem~\ref{De_inverse.thm} shows that
$\tDe$ is invertible.
\end{remark}

\section{Problems with a single boundary condition}
\label{single_bc.sec}

In this section we consider the following problem with a single,
multi-point boundary condition,
\begin{gather}
 -u'' = \la u ,  \quad \text{on $\R$},
\label{single_bc_de.eq}
\\
\al_0 u(\eta_0)  + \be_0 u'(\eta_0) =
\sum_{i=1}^m \al_i u(\eta_i)
 +
\sum_{i=1}^m \be_i u'(\eta_i) ,
\label{single_bc.eq}
\end{gather}
where
$m \ge 1$,
$\al_0,\be_0,\eta_0 \in \R$,
and
$\al,\be,\eta \in \R^m$.
The conditions \eqref{albe_nz.eq} and \eqref{AB_lin_cond.eq}
have obvious analogues in the current setting, simply by omitting the
superscripts $\pm$,
which we will use without further comment,
while the condition \eqref{albe_sign.eq} has no analogue here and
$\be_0$ may have either sign.
For any $(\al_0,\be_0)$ satisfying
\eqref{albe_nz.eq}
we let $\B(\al_0,\be_0)$ denote the set of
$(\al,\be) \in \R^{2m}$ satisfying \eqref{AB_lin_cond.eq}.

\begin{thm}  \label{single_bc.thm}
Suppose that
$(\al_0,\be_0,\al,\be)$ satisfies
\eqref{albe_nz.eq} and \eqref{AB_lin_cond.eq},
and $\la \ge 0$.
Then the set of solutions of
\eqref{single_bc_de.eq}, \eqref{single_bc.eq},
is one-dimensional.
\end{thm}

\begin{proof}
If $\la=0$ then any solution of \eqref{single_bc_de.eq} has the form of
$u_0$ used in the proof of Theorem~\ref{De_inverse.thm},
and substituting $u_0$ into \eqref{single_bc.eq} yields a linear
equation relating the coefficients $c_0,c_1$.
A similar argument to the proof of Theorem~\ref{De_inverse.thm} now
shows that the set of solutions of this equation is one-dimensional.

Now suppose that $\la > 0$.
For any $s > 0$, $\th \in \R$, we define
$w(s,\th) \in C^1(\R)$ by
\begin{equation}  \label{s_soln_form.eq}
w(s,\th)(x) := \sin(s x+\th) , \quad x \in \R .
\end{equation}
Clearly, any solution of \eqref{single_bc_de.eq}
must have the form $u = C w(s,\th)$,
with $s = \la^{1/2}$ and suitable $C,\,\th \in \R$.
For the rest of this proof we regard $\th$, $\al$, $\be$ as variable,
but all the other parameters and coefficients will be regarded as fixed
and omitted from the notation when this is convenient.
Defining $\Ga : \R \X \R^{2m} \to \R$ by
\begin{align*}
\Ga(\th,\al,\be)  :=
\al_0 &\sin(s \eta_0 +\th)  + s \be_0 \cos(s \eta_0 +\th)
 - \sum_{i=1}^m \al_i \sin(s\eta_i+\th)
\\
&
 - s \sum_{i=1}^m \be_i \cos(s\eta_i+\th) ,
\end{align*}
it is clear that $\Ga$ is $C^1$,
and substituting \eqref{s_soln_form.eq} into \eqref{single_bc.eq}
shows that $w(s,\th)$ satisfies
\eqref{single_bc_de.eq}, \eqref{single_bc.eq} if and only if
\begin{equation}  \label{single_bc_zeros.eq}
\Ga(\th,\al,\be) = 0 .
\end{equation}
Hence, it suffices to consider the set of
solutions of \eqref{single_bc_zeros.eq}.

Next, by definition, for any $(\al,\be) \in \R^{2m}$
the function $\Ga(\cdot,\al,\be)$ is $\pi$-antiperiodic,
so to prove the theorem it suffices to show that
if $(\al,\be) \in \B(\al_0,\be_0)$ then $\Ga(\cdot,\al,\be)$
has exactly one zero in the interval $[0,\pi)$
(by $\pi$-antiperiodicity, other zeros of $\Ga(\cdot,\al,\be)$ do not
contribute distinct solutions of
\eqref{single_bc_de.eq}, \eqref{single_bc.eq}).
We will prove this by a continuation argument.

We first observe that if $(\al,\be) = (0,0)$ then
$\Ga(\cdot,0,0)$ has exactly 1 zero in $[0,\pi)$ and this zero is
simple.
To extend this property to $(\al,\be) \ne (0,0)$ we will require the
following lemma
($\Ga_\th$ will denote the partial derivative of $\Ga$ with respect to
$\th$).

\begin{lemma}  \label{single_bc_simple.lem}
Suppose that
$(\al_0,\be_0,\al,\be)$ satisfies
\eqref{albe_nz.eq} and \eqref{AB_lin_cond.eq},
and
$\la > 0$.
Then
$$
\Ga(\th,\al,\be) = 0  \implies  \Ga_\th(\th,\al,\be) \ne 0.
$$
\end{lemma}

\begin{proof}
Suppose, on the contrary, that
\begin{equation}  \label{single_bc_Ga_zeros_simple.eq}
\Ga(\th,\al,\be) = \Ga_\th(\th,\al,\be) = 0,
\end{equation}
for some $\th \in \R$ and $(\al,\be) \in \B(\al_0,\be_0)$.
We now regard $(\th,\al,\be)$ as fixed, and write
$$
S(\eta) := \sin(s\eta+\th),
\quad
C(\eta) := \cos(s\eta+\th) .
$$
With this notation, equations
\eqref{single_bc_Ga_zeros_simple.eq}
become
\begin{align}
\al_0 S(\eta_0) + s \be_0 C(\eta_0)
&= \sum_{i=1}^m \big( \al_i S(\eta_i) + s \be_i C(\eta_i) \big) ,
\label{Ga_zeros_lin_a.eq}
\\
\al_0 C(\eta_0) - s \be_0 S(\eta_0)
&= \sum_{i=1}^m \big( \al_i C(\eta_i) - s \be_i S(\eta_i) \big) .
\label{Ga_zeros_lin_b.eq}
\end{align}
By \eqref{AB_lin_cond.eq} we can choose  $b_0 \in [0 , \pi/2]$ such
that, with $C_b := \cos b_0$, $S_b := \sin b_0$,
\begin{equation}  \label{AB_lin_cond_sc.eq}
\sum^m_{i=1} |\al_i| \le C_b \al_0  ,
\quad
\sum^m_{i=1} |\be_i| \le S_b |\be_0|  ,
\end{equation}
with at least one strict inequality in \eqref{AB_lin_cond_sc.eq}.

Now suppose that $\be_0 \ge 0$.
Elementary operations on
\eqref{Ga_zeros_lin_a.eq}, \eqref{Ga_zeros_lin_b.eq} now yield
\begin{align*}
&C_b \al_0 +  S_bs \be_0 =
\\[1 ex]
& =
\sum_{i=1}^m \al_i
\Bigl(
C_b ( S(\eta_0) S(\eta_i) + C(\eta_0) C(\eta_i))
+
S_b ( C(\eta_0) S(\eta_i) - S(\eta_0) C(\eta_i))
\Bigr)
\\[1 ex]
& \quad + s \sum_{i=1}^m \be_i
\Bigl(
C_b ( S(\eta_0) C(\eta_i) - C(\eta_0) S(\eta_i))
+
S_b ( C(\eta_0) C(\eta_i) + S(\eta_0) S(\eta_i))
\Bigr)
\\[1 ex]
& =
\sum_{i=1}^m \al_i
\Bigl(
C_b \cos s(\eta_0-\eta_i) - S_b \sin s(\eta_0 - \eta_i)
\Bigr)
\\[1 ex]
& \quad
 + s \sum_{i=1}^m \be_i
\Bigl(
C_b \sin s(\eta_0 - \eta_i) + S_b \cos s(\eta_0 - \eta_i)
\Bigr)
\\[1 ex]
& =
\sum_{i=1}^m \al_i  \cos s(b_0 + \eta_0 - \eta_i)
+
s \sum_{i=1}^m \be_i  \sin s(b_0 - \eta_0 + \eta_i)
\\[1 ex]
&
\le
\sum_{i=1}^m |\al_i|
+
s \sum_{i=1}^m |\be_i|
 <
C_b \al_0 +  S_b s \be_0 ,
\end{align*}
by \eqref{AB_lin_cond_sc.eq}.
This contradiction shows that \eqref{single_bc_Ga_zeros_simple.eq}
cannot hold, and so proves the lemma,
when $\be_0 \ge 0$.
If $\be_0 < 0$ then we simply replace
$C_b \al_0 +  S_b s \be_0$
with
$C_b \al_0 - S_b  s \be_0$
in the above calculation to obtain a similar contradiction,
which completes the proof of Lemma~\ref{single_bc_simple.lem}.
\end{proof}

Now, since the set  $\B(\al_0,\be_0)$ is
connected it follows from continuity, together with
Lemma~\ref{single_bc_simple.lem}, the implicit function theorem  and the
$\pi$-antiperiodicity of $\Ga(\cdot,\al,\be)$, that
$\Ga(\cdot,\al,\be)$ has exactly 1 (simple) zero in $[0,\pi)$
for all $(\al,\be) \in  \B(\al_0,\be_0)$.
This completes the proof of Theorem~\ref{single_bc.thm}.
\end{proof}

For Dirichlet-type and Neumann-type problems,
Theorem~\ref{single_bc.thm} was proved in \cite{RYN5} and \cite{RYN6},
respectively.
An adaptation of the proof of
Lemma~\ref{single_bc_simple.lem}
also yields the following result, which will be crucial below.

\begin{lemma}  \label{bcerp.lem}
Suppose that
$\la > 0$
and $(\al_0,\be_0,\al,\be)$ satisfies
\eqref{albe_nz.eq} and \eqref{AB_lin_cond.eq}.
If $u$ is a non-trivial solution of
\eqref{single_bc_de.eq}, \eqref{single_bc.eq} then
\begin{equation}  \label{bcnonperp.eq}
\la \be_0 u(\eta_0)  - \al_0 u'(\eta_0) \ne 0.
\end{equation}
\end{lemma}

\begin{proof}
The argument is similar to the proof of
Lemma~\ref{single_bc_simple.lem},
and we use the notation from there.
In particular, we suppose that $u$ has the form of $w$ given in
\eqref{s_soln_form.eq}, so that \eqref{single_bc.eq} takes the form
\eqref{Ga_zeros_lin_a.eq},
and to obtain a contradiction we suppose that \eqref{bcnonperp.eq}
fails, that is, with this form of $u$,
\begin{equation}  \label{bcperp.eq}
s \be_0 S(\eta_0)  - \al_0 C(\eta_0) = 0.
\end{equation}
Multiplying  \eqref{Ga_zeros_lin_a.eq} by $S(\eta_0)$ and $C(\eta_0)$,
and using \eqref{bcperp.eq}, yields respectively
\begin{align*}
\al_0
&=
S(\eta_0) \sum_{i=1}^m \big( \al_i S(\eta_i) + \be_i s C(\eta_i) \big) ,
\\
s \be_0
&=
C(\eta_0)
  \sum_{i=1}^m \big( \al_i S(\eta_i) + \be_i s C(\eta_i) \big) .
\end{align*}
If $\be_0 \ge 0$ then
combining these inequalities and using \eqref{AB_lin_cond_sc.eq} yields
\begin{align*}
C_b \al_0 + S_b s \be_0
& =
\big( C_b S(\eta_0) + S_b C(\eta_0) \big)
\sum_{i=1}^m \big( \al_i S(\eta_i) + \be_i s C(\eta_i) \big)
\\ & <
C_b \al_0 + S_b s \be_0 ,
\end{align*}
which is the desired contradiction in this case.
If $\be_0 < 0$ then we simply replace
$C_b \al_0 + S_b s \be_0$
with
$C_b \al_0 - S_b s \be_0$
in the preceding calculation to obtain a similar contradiction.
This completes the proof of Lemma~\ref{bcerp.lem}.
\end{proof}

We also have the following immediate application of
Theorem~\ref{single_bc.thm} to the eigenvalue problem.

\begin{cor}  \label{geom_mult_one.cor}
Suppose that
$(\al_0^\pm,\be_0^\pm,\al^\pm,\be^\pm)$ satisfy
\eqref{albe_nz.eq} and \eqref{AB_lin_cond.eq}.
Then any eigenvalue $\la > 0$ of
\eqref{eval_de.eq},  \eqref{slbc.eq},
has geometric multiplicity one.
\end{cor}

\subsection{Counter examples} \label{vip_counterexamples.sec}

The following example shows that if $\la < 0$ then
Theorem~\ref{single_bc.thm} need not hold.

\begin{ex}  \label{single_condn_la_neg_SL.ex}
Consider \eqref{single_bc_de.eq} with $\la = -1$,
together with the boundary condition
\begin{equation}  \label{single_condn_la_neg_ex_one_SL.eq}
u(-1) + u'(-1) = \al_1 u(0) + \be_2 u'(1).
\end{equation}
that is, with
$\al_0 = \be_0 = 1$, $\be_1 = \al_2 = 0$ and
$\eta_0 = -1$, $\eta_1 = 0$, $\eta_2 = 1$;
we will choose $\al_1$ and $\be_2$ below.
The general solution of equation \eqref{single_bc_de.eq} is
$u(x) = c_+ e^x + c_- e^{-x}$,
for arbitrary $(c_+,c_-) \in \R^2$,
and substituting this solution into the boundary condition
\eqref{single_condn_la_neg_ex_one_SL.eq} yields the equation
\begin{equation}  \label{single_condn_la_neg_ex_two_SL.eq}
c_+ (2 - \al_1 e  - \be_2 e^2)
 -
c_- (\al_1 e^{-1}  - \be_2 e^{-2})
= 0.
\end{equation}
Now, setting
$$
\al_1 = \frac{2}{e(e^2 +1)}, \quad \be_2 = \frac{2}{e^2 +1},
$$
we see that
$(\al_0,\be_0,\al,\be)$ satisfies
\eqref{AB_lin_cond.eq},
and \eqref{single_condn_la_neg_ex_two_SL.eq} holds for all
$(c_+,c_-) \in \R^2$.
Hence, the solution set of this boundary value problem is
two-dimensional, and so Theorem~\ref{single_bc.thm} does not hold in
this case.
\hfill $\square$ \end{ex}

Example~\ref{single_condn_la_neg_SL.ex} can be extended to the
eigenvalue problem to show that Corollary~\ref{geom_mult_one.cor} need
not hold for negative eigenvalues.

\begin{ex}  \label{dble_condn_la_neg_SL.ex}
Consider the multi-point eigenvalue problem consisting of
equation \eqref{single_bc_de.eq}
together with the pair of boundary conditions
\begin{equation}  \label{dble_condn_la_neg_SL.eq}
u(\pm 1) \mp u'(\pm 1) = \al_1 u(0) \mp \be_2 u'(\mp 1) ,
\end{equation}
with $\al_1$ and $\be_2$ as in Example~\ref{single_condn_la_neg_SL.ex}.
It can be verified (as in Example~\ref{single_condn_la_neg_SL.ex})
that $\la = -1$ is an eigenvalue of this boundary value problem with
geometric multiplicity two.
Hence, Corollary~\ref{geom_mult_one.cor} need not hold for negative
eigenvalues.
We observe that both sets of boundary condition coefficients in this
problem satisfy \eqref{AB_lin_cond.eq},
but of course the sign condition \eqref{albe_sign.eq} does not hold
(which allows the negative eigenvalue).
\hfill $\square$ \end{ex}

The final example in this section shows that if $\la < 0$ then
Theorem~\ref{single_bc.thm} need not hold,
even with a Dirichlet-type boundary condition
(that is, with $\be_0 = 0$ and $\be = 0$).
However, this example is not relevant to the eigenvalue problem since
negative eigenvalues do not occur with Dirichlet-type
boundary conditions
(also, in this example $\eta_1 < \eta_0 < \eta_2$,
which is not consistent with the distribution of these points in the
eigenvalue problem).

\begin{ex}  \label{single_condn_la_neg_D.ex}
Consider  \eqref{single_bc_de.eq} with $\la = -1$,
together with the boundary condition
\begin{equation}  \label{single_condn_la_neg_ex_one_D.eq}
u(0) = \frac{e(e^2-1)}{e^4-1} \big( u(-1) + u(1) \big) .
\end{equation}
It can be verified that
\eqref{AB_lin_cond.eq} again holds,
and for arbitrary $(c_+,c_-) \in \R^2$,
the function $u(x) = c_+ e^x + c_- e^{-x}$
satisfies both
\eqref{single_bc_de.eq} and
\eqref{single_condn_la_neg_ex_one_D.eq},
that is the solution set of this boundary value problem is again
two-dimensional.
\hfill $\square$ \end{ex}

\section{The structure of $\si$}  \label{evals.sec}

In this section we discuss the structure of the spectrum of
the multi-point eigenvalue problem
\eqref{eval_de.eq}-\eqref{slbc.eq},
which we can rewrite as
\begin{equation}  \label{mp_eval.eq}
 -\De(u)  = \la u , \quad u \in X .
\end{equation}
We will show that $\si$ has the properties ($\si$-a)-($\si$-d)
described in the introduction, that is, the multi-point spectrum has
similar properties to the spectrum of the standard Sturm-Liouville with
separated boundary conditions.
In particular, we will obtain a characterisation of the
eigenvalues in terms of an oscillation count of the corresponding
eigenfunctions, as in the property ($\si$-d) in the introduction.

The standard method of counting the oscillations of the eigenfunctions
of separated problems is by counting the number of (nodal) zeros in the
interval $(-1,1)$, and it is well known that this approach yields
property ($\si$-d) in this case.
Unfortunately, this need not be true for the multi-point boundary
conditions.
This was first observed in \cite{RYN3}, in the case of a
problem with a single-point Dirichlet condition at one end
point and a multi-point Dirichlet-type
condition at the other end point.
For such a problem it was shown that, for $k \ge 0$,
if $u_k$ is an eigenfunction corresponding to $\la_k$
then $u_k$ could have either $k$ or $k+1$ zeros in $(-1,1)$,
whereas $u_k'$ has exactly $k+1$ zeros in $(-1,1)$
(these zeros of $u_k'$ were were termed `bumps' in  \cite{RYN3}).
The results of \cite{RYN3} were then extended to a similar $p$-Laplacian
problem  in \cite{DR}, and $p$-Laplacian problem with multi-point
Dirichlet-type conditions at both end points in \cite{RYN5}.
Thus, in the  Dirichlet-type case, using nodal zeros to count the
eigenfunction oscillations fails, and in fact the oscillations are best
described by counting bumps
(and by starting the enumeration of the eigenvalues/eigenfunctions at
$k=1$, that is, the first eigenfunction has a single bump).

However, it was then shown in \cite{RYN6} that counting bumps fails in
the case of Neumann and mixed boundary conditions, and in fact in
\cite{RYN3, RYN5, RYN6} a
different oscillation counting procedure was adopted for each of these
three types of boundary conditions, and each of these procedures could
fail when applied to the other problems.
To deal with the general Sturm-Liouville-type boundary conditions here
we will use a Pr\"ufer angle technique to characterise the
oscillation count of the eigenfunctions.
This technique will unify and extend the various types of oscillation
count used previously in \cite{RYN3, RYN5, RYN6}.

In view of this we begin with a preliminary section discussing a
Pr\"ufer angle method of defining an oscillation count for the
multi-point problem.
We then use this oscillation count to describe the multi-point spectrum.

\subsection{Pr\"ufer angles and oscillation count} \label{Neu_nodal.sec}

The Pr\"ufer angle is a standard technique in the theory of ordinary
differential equations, although there are slight variations in
the precise definitions and functions used.
The basic formulation is described in \cite[Chapter 8]{CL}
(although the terminology `Pr\"ufer angle' is not used in \cite{CL}).
However, a more general formulation is described in
\cite[Section~2]{BD} (in a $p$-Laplacian context),
together with some remarks about  various  `modified Pr\"ufer angle'
formulations, and their history.
In fact, we will adopt the  form of the angle used in
\cite[Lemma 2.5]{BD},
which was used earlier by Elbert
(see Remark~\ref{mod_prufer_angle.rem} below for the reason for our
use of this formulation).
We will then see that, in contrast to the separated case,
the multi-point boundary conditions \eqref{slbc.eq} do not determine
the exact values of the Pr\"ufer angle at the end points $\pm 1$,
but instead they place bounds on these angles.

We will give a full description of our constructions and results
relating to the boundary conditions \eqref{slbc.eq} but, for brevity, we
will not describe the basic details of the Pr\"ufer angle technique here
but simply refer to \cite{BD} and \cite{CL}.

Let $\Cs[-1,1]$ denote the set of functions $u \in C^1[-1,1]$
having only simple zeros
(that is, $|u(x)| + |u'(x)| > 0$ for all $x \in [-1,1]$).
For any $\la > 0$ and $u \in \Cs[-1,1]$, we define a `modified'
Pr\"ufer angle function
$\om_{(\la,u)} \in C^0[-1,1]$ by
\begin{equation} \label{pruf_defn.eq}
\om_{(\la,u)}(-1) \in [0,\pi),
\quad
\om_{(\la,u)}(x) := \tan^{-1} \frac {\la^{1/2} u(x)}{u'(x)}, \quad x \in
[-1,1]
\end{equation}
(when $u'(x) = 0$ the value of $\om_{(\la,u)}(x)$ is defined by
continuity).
We note that the standard Pr\"ufer angle does not have the factor
$\la^{1/2}$ in the definition.
Geometrically, for each $x \in [-1,1]$ we can regard $\om_{(\la,u)}(x)$
as the angle between
the vectors $(u'(x),\la^{1/2} u(x))$ and $(1,0)$ in $\R^2$,
defined to vary continuously with respect to $x$
(so $\om_{(\la,u)}(x)$ need not lie within $[0,\pi/2]$, or even within
$[0,2\pi]$).
Clearly, if $u$ is a non-trivial solution of the
differential equation \eqref{eval_de.eq} then $u \in \Cs[-1,1]$,
so $\om_{(\la,u)}$ is well defined.

From now on we suppose that
\eqref{albe_nz.eq} and \eqref{albe_sign.eq} hold,
and we also define the angles
\begin{align*}
\om_{\la,0}^- := - \tan^{-1} \frac {\la^{1/2} \be_0^-}{\al_0^-}
  \in [0,\pi/2] ,
\qquad
\om_{\la,0}^+  := - \tan^{-1} \frac {\la^{1/2} \be_0^+}{\al_0^+}
  \in [\pi/2,\pi] ,
\end{align*}
where the permissible ranges chosen here for the values of
$\om_{\la,0}^\pm$
are consistent with the sign conditions \eqref{albe_sign.eq}.
Geometrically,  $\om_{\la,0}^\pm$
are the angles between
the vectors $(\al_0^\pm, - \la^{1/2} \be_0^\pm)$ and $(1,0)$.

\subsubsection{Suppose that $(\bal,\bbe) = (\bzero,\bzero)$.}
In this case the boundary conditions
\eqref{slbc.eq} reduce to the separated conditions
\begin{equation} \label{slbc_zero.eq}
\la^{-1/2} (u'(\pm 1),\la^{1/2} u(\pm 1)) .
(\la^{1/2} \be_0^\pm,\al_0^\pm) =
\al_0^\pm u(\pm 1)  + \be_0^\pm u'(\pm 1) = 0
\end{equation}
(where the left hand side is the usual dot product of the vectors).
That is, a function $u \in C^1[-1,1]$ satisfies
\eqref{slbc.eq} if and only if
\begin{equation}  \label{vectors_perp.eq}
\text{$(u'(\pm 1),\la^{1/2} u(\pm 1))$ is perpendicular to
$(\la^{1/2} \be_0^\pm,\al_0^\pm)$.}
\end{equation}
Since the vectors
$(\al_0^\pm, - \la^{1/2} \be_0^\pm)$
and
$(\la^{1/2} \be_0^\pm,\al_0^\pm)$
are perpendicular, we see that $u$ satisfies
\eqref{vectors_perp.eq} if and only if
\begin{equation}  \label{vectors_parallel.eq}
\text{$(u'(\pm 1),\la^{1/2} u(\pm 1))$ is parallel to
$(\al_0^\pm, - \la^{1/2} \be_0^\pm)$,}
\end{equation}
which is equivalent to
\begin{equation} \label{pruf_eq.eq}
\om_{(\la,u)}(\pm 1) = \om_{\la,0}^\pm \ (\rm{mod}\ \pi) .
\end{equation}
Standard Sturm-Liouville theory for the
separated boundary conditions \eqref{slbc_zero.eq} now yields the
following properties of the spectrum,
see Theorem~2.1 in  \cite[Chapter~8]{CL}
(and the proof of this theorem).

\begin{thm}  \label{spec_zero.thm}
Suppose that $(\bal,\bbe) = (\bzero,\bzero)$.
Then $\si$ consists of a strictly increasing sequence of real
eigenvalues
$\la_k^\bzero \ge 0$, $k=0,1,\dots.$
For each $k \ge 0$$:$
\begin{mylist}
\item
$\la_k^\bzero$ has geometric multiplicity one$;$
\item
$\la_k^\bzero$ has an eigenfunction $u_k^\bzero$
whose Pr\"ufer angle $\om_k^\bzero := \om_{(\la_k^\bzero,u_k^\bzero)}$
satisfies
\begin{equation} \label{om_efun_zero.eq}
\om_k^\bzero(-1) = \om_{\la,0}^-,
\quad
\om_k^\bzero(1) = \om_{\la,0}^+ + k \pi .
\end{equation}
\end{mylist}
\end{thm}

\begin{remark} \label{oscillations_zero.rem}
By definition, for any $u \in \Cs[-1,1]$,
\begin{align*}
u(x) &= 0 \iff \om_{(\la,u)}(x) = 0\, (\rm{mod}\, \pi) ,
\\
u'(x) &= 0 \iff \om_{(\la,u)}(x) = \frac{\pi}{2}\, (\rm{mod}\, \pi) .
\end{align*}
In addition, it can be verified that if $u$ satisfies the differential
equation \eqref{eval_de.eq}, with $\la > 0$, then
$$
u(x) u'(x) = 0 \implies \om'_{(\la,u)}(x) > 0 ,
$$
so it follows from \eqref{om_efun_zero.eq}
that, for all $k \ge 0$, the eigenfunction $u_k^\bzero$
has exactly $k$ zeros in the interval $(-1,1)$;
this is the usual `oscillation count' for the standard, separated,
Sturm-Liouville problem.
Thus the oscillation count of the eigenfunctions of the separated
problem can be described by the Pr\"ufer angle, and this count is
encapsulated in \eqref{om_efun_zero.eq}.
\end{remark}

\subsubsection{Suppose that
$(\bzero,\bzero) \ne (\bal,\bbe) \in \B(\bal_0,\bbe_0)$.}

In this case the eigenfunctions need not satisfy
\eqref{vectors_parallel.eq}-\eqref{om_efun_zero.eq} --- to provide a
replacement for these formulae we first prove the following lemma.

\begin{lemma} \label{pruf_bnd.lem}
Suppose that $u$ is an eigenfunction, with eigenvalue $\la > 0$.
Then
\begin{equation} \label{pruf_bnd.eq}
\om_{(\la,u)}(\pm 1) - \om_{\la,0}^\pm  \ne \frac{\pi}{2}
\ (\rm{mod}\, \pi) .
\end{equation}
\end{lemma}

\begin{proof}
It follows from the definitions of $\om_{(\la,u)}$ and the angles
$\om_{\la,0}^\pm$
that
$$
\om_{(\la,u)}(\pm 1)-\om_{\la,0}^\pm = \frac{\pi}{2} \ ({\rm mod}\,
\pi)
\iff
\la \be_0^\pm u(\pm 1) - \al_0^\pm u'(\pm 1) = 0 ,
$$
so the result follows from Lemma~\ref{bcerp.lem}
(by putting $\eta_0 = \pm 1$, etc.).
\end{proof}

The geometrical interpretation of \eqref{pruf_bnd.eq} is:
\begin{equation}  \label{vectors_non_parallel.eq}
\text{$(u'(\pm 1),\la^{1/2} u(\pm 1))$
is not perpendicular to
$(\al_0^\pm,-\la^{1/2} \be_0^\pm)$.
}
\end{equation}
Thus we see that going from separated to  multi-point boundary
conditions has relaxed the `strictly parallel' condition
\eqref{vectors_parallel.eq}, holding in the separated case, to the
`not perpendicular' condition \eqref{vectors_non_parallel.eq},
holding in the multi-point case.

Motivated by Theorem~\ref{spec_zero.thm} and Lemma~\ref{pruf_bnd.lem},
we introduce some further notation.

\begin{definition}
For  $k \ge 0$,
$P_k^+$ will denote the set of
$(\la,u) \in (0,\infty) \X \Cs[-1,1]$ for which the Pr\"ufer
angle $\om_{(\la,u)}$ satisfies
\begin{equation} \label{om_efun_gen.eq}
|\om_{(\la,u)}(-1) - \om_{\la,0}^-| < \pi/2 ,
\quad
|\om_{(\la,u)}(1) - \om_{\la,0}^+ - k\pi| < \pi/2 ;
\end{equation}
also, $P_k^- := - P_k^+$ and $P_k :=  P_k^- \cup P_k^+$.
\end{definition}

The sets $P_k^\pm$, $k \ge 0$, are
open, disjoint subsets of $(0,\infty) \X C^1[-1,1]$,
and they will be used to count
eigenfunction oscillations in Theorem~\ref{spec.thm} below.
In fact, the results of Theorem~\ref{spec.thm} below will demonstrate
that,
for general $(\bal,\bbe) \ne (\bzero,\bzero)$,
the conditions
\eqref{pruf_bnd.eq}
and
\eqref{om_efun_gen.eq}
are suitable replacements for conditions
\eqref{pruf_eq.eq}
and
\eqref{om_efun_zero.eq}
respectively.
As a preliminary to this we observe that the above definitions,
together with  Corollary~\ref{geom_mult_one.cor} and
Lemma~\ref{pruf_bnd.lem} yield the following result.

\begin{cor} \label{efun_in_P.cor}
Suppose that $u$ is an eigenfunction, with eigenvalue $\la > 0$.
Then$:$
\begin{mylist}
\item
$\la$ has geometric multiplicity $1;$
\item
$(\la,u) \not\in \pa P_l$, for any  $l
\ge 0;$
\item
there exists  $k \ge 0$ such that $(\la,u) \in P_k$.
\end{mylist}
\end{cor}

Motivated by Corollary~\ref{efun_in_P.cor} we define the sets
$$
\si_k := \{ \la \in \si :
\text{for any eigenfunction $u$ of $\la$, $(\la,u) \in P_k$}\},
\quad k \ge 0.
$$
By Corollary~\ref{efun_in_P.cor}, we have $\si = \cup_{k \ge 0}
\,\si_k$.

\begin{remark}  \label{oscillations_gen.rem}
In \cite{RYN5} and \cite{RYN6} certain subsets of
$\Cs[-1,1]$, denoted $T_k$ and $S_k$, were used to count oscillations in
the Dirichlet-type and Neumann-type
cases respectively.
It follows from the results in Remark~\ref{oscillations_zero.rem}
and the definitions of $T_k$ and $S_k$ in
\cite[Section~2.2]{RYN5} and \cite[Section~2.2]{RYN6}
that, for each integer $k \ge 0$:
\begin{itemize}
\item
Neumann-type case: \quad
$\om_{\la,0}^\pm = \frac{\pi}{2}$
and\\[.5 ex]
$
(\la,u) \in P_k \implies \text{$u$ has exactly
$k$ zeros in $(-1,1)$
and $u \in S_k$;}
$
\smallskip
\item
Dirichlet-type case:  \quad
$\om_{\la,0}^- = 0$, \
$\om_{\la,0}^+ = \pi$
and\\[.5 ex]
$
(\la,u) \in P_k \implies \text{$u'$ has exactly
$k+1$ zeros in $(-1,1)$
and $u \in T_{k+1}$.}
$
\smallskip
\end{itemize}
Hence,
in the Dirichlet-type and Neumann-type cases respectively,
the sets $P_k$ used here are analogous to the sets
$(0,\infty) \X T_{k+1}$ and $(0,\infty) \X S_k$,
and we see that using the sets $P_k$ to count the eigenfunction
oscillations extends the oscillation counting methods used in the
above special cases to the general Sturm-Liouville-type boundary
conditions considered here.
\end{remark}

\begin{remark}  \label{mod_prufer_angle.rem}
The above constructions
depended on \eqref{bcnonperp.eq},
via Lemma~\ref{pruf_bnd.lem},
and the occurrence of the term $\la^{1/2}$ in
\eqref{bcnonperp.eq} dictated that the
term $\la^{1/2}$ should appear in the definition of the Pr\"ufer angle.
This is why we have used the `modified' Pr\"ufer angle here.
\end{remark}

\subsection{The structure of  $\si$}   \label{spec.sec}

We can now prove the following theorem for general $(\bal,\bbe)$,
which extends Theorem~\ref{spec_zero.thm} to the
general multi-point Sturm-Liouville problem.

\begin{thm}  \label{spec.thm}
Suppose that
\eqref{albe_nz.eq}-\eqref{AB_lin_cond.eq}
hold.
Then $\si$ consists of a strictly increasing sequence of real
eigenvalues
$\la_k \ge 0$, $k=0,1,\dots,$
such that $\lim_{k \to \infty} \la_k = \infty.$
For each $k \ge 0$$:$
\begin{mylist}
\item
$\la_k$ has geometric multiplicity $1;$
\item
$\la_k$ has an eigenfunction $u_k$
such that $(\la_k,u_k) \in P_k^+$.
\end{mylist}
In the Neumann-type case $\la_0 = 0$, while if
\eqref{al_pm_strict_pos.eq} holds then $\la_0 > 0$.
\end{thm}

\begin{proof}

We will prove a series of results regarding the eigenvalues
and eigenfunctions,
which culminate in the proof of the theorem.
The fact that the eigenvalues have geometric multiplicity 1 has already
been proved in Corollary~\ref{geom_mult_one.cor}.

\begin{lemma}  \label{la_ge_zero.lem}
If $\la$ is an eigenvalue then $\la \ge 0$.
If  \eqref{al_pm_strict_pos.eq} holds then $\la > 0$.
\end{lemma}

\begin{proof}
Suppose that $\la < 0$ and define $s := \sqrt{-\la}$.
Then any eigenfunction $u$
has the form $u(x) = c_+ e^{sx} + c_- e^{-sx}$,
for some $(c_+,c_-) \in \R^2$,
and we see from this that
$\max |u|$ and  $\max |u'|$
must both be attained at the same end point, say at $x=1$.
Hence, $u(1)$ and $u'(1)$ have the same sign.
By \eqref{albe_sign.eq}, $\be_0^+ \ge 0$, so by
\eqref{slbc.eq} and \eqref{AB_lin_cond.eq},
\begin{align*}
\al_0^+ |u|_0  + \be_0^+ |u'|_0
&=
|\al_0^+ u(1)  + \be_0^+ u'(1)|
\\ & \le
|u|_0  \sum_{i=1}^{m^+} |\al^+_i |
 +
|u'|_0 \sum_{i=1}^{m^+} |\be^+_i|
\\ & <
\al^+_0 |u|_0 + \be^+_0 |u'|_0  ,
\end{align*}
and this contradiction proves the first part of the lemma.
Next, if  \eqref{al_pm_strict_pos.eq} holds then it follows from
Theorem~\ref{De_inverse.thm} that $\la \ne 0$, which completes the
proof.
\end{proof}

\begin{remark}  \label{la_ge_zero.rem}
It is well known that if the sign conditions \eqref{albe_sign.eq} do not
hold then
Lemma~\ref{la_ge_zero.lem} need not be true, even in the separated
case.
For example, if
$$
\al_0^\pm = \pm \ep, \quad \al = 0,
\qquad
\be_0^\pm = \pm 1, \quad \be = 0 .
$$
\end{remark}

The properties of the spectrum in the Neumann-type case have been
proved in \cite{RYN6}, so from now on in the proof we will suppose that
\eqref{al_pm_strict_pos.eq} holds.
Thus, by
Theorem~\ref{De_inverse.thm} and Lemma~\ref{la_ge_zero.lem},
if $\la$ is an eigenvalue with eigenfunction $u$, then $\la > 0$
and we may suppose that
$\la = s^2 ,$ $u = w(s,\th)$,
for suitable $s > 0$, $\th \in \R$
(up to a scaling of the eigenfunction),
where $w(s,\th)$ was defined in \eqref{s_soln_form.eq}.
Defining functions
$\Ga^\pm : (0,\infty) \X \R \X \R^{2(m^-+m^+)} \to \R$
by
\begin{equation*}
\begin{split}
\Ga^\pm(s,\th,\al^\pm,\be^\pm)  &:=
\al_0^\pm \sin(\pm s+\th)  + s \be_0^\pm \cos'(\pm s+\th)\, -
\\
& \quad - \sum^{m^\pm}_{i=1} \al_i^\pm \sin(s\eta_i^\pm+\th)
 - s \sum^{m^\pm}_{i=1} \be_i^\pm \cos(s\eta_i^\pm+\th) ,
\end{split}
\end{equation*}
and substituting $w(s,\th)$ into \eqref{slbc.eq}
shows that
$\la = s^2$
is an eigenvalue iff
the pair of equations
\begin{equation}   \label{simul_zeros.eq}
\Ga^\pm(s,\th,\al^\pm,\be^\pm) = 0
\end{equation}
holds, for some $\th \in \R$.
Hence, it suffices to consider the set of solutions of
\eqref{simul_zeros.eq}.

We will now prove Theorem~\ref{spec.thm} by continuation with respect
to $(\bal,\bbe)$, away from
$(\bal,\bbe)=(\bzero,\bzero)$, where the required information on the
solutions of \eqref{simul_zeros.eq} follows from the standard theory of
the separated problem in  Theorem~\ref{spec_zero.thm}.
For reference, we state this in the following lemma.

\begin{lemma}  \label{Ga_zeros_at_zero.lem}
Suppose that $(\bal,\bbe) = (\bzero,\bzero)$.
For each $k = 0,1,\dots,$
if we write $s_k^\bzero := (\la_k^\bzero)^{1/2}$
$($where $\la_k^\bzero$ is as in
Theorem~\ref{spec_zero.thm}$)$,
then there exists a unique
$\th_k^\bzero \in [0,\pi)$ such that
$(s_k^\bzero,\th_k^\bzero)$
satisfies \eqref{simul_zeros.eq}.
\end{lemma}

Of course, by the periodicity properties of $\Ga^\pm$ with respect to
$\th$, there are other solutions of \eqref{simul_zeros.eq}
(with $(\bal,\bbe) = (\bzero,\bzero)$)
than those in Lemma~\ref{Ga_zeros_at_zero.lem},
but these do not yield distinct solutions of
the eigenvalue problem \eqref{mp_eval.eq}.
In fact, to remove these extra solutions and to reduce the domain of
$\th$ to a compact set, from now on we will regard $\th$ as lying in
the circle obtained from the interval $[0,2\pi]$ by identifying the
points $0$ and $2\pi$,
which we denote by $S^1$,
and we regard the domain of the functions $\Ga^\pm$ as
$(0,\infty) \X S^1 \X \B(\al_0^\pm,\be_0^\pm)$.

We now consider \eqref{simul_zeros.eq}  when
$(\bal,\bbe) \ne (\bzero,\bzero)$.
The following proposition provides some information on the signs of the
partial derivatives $\Ga^\nu_s$, $\Ga^\nu_\th$ at the zeros of
$\Ga^\nu$.

\begin{lemma} \label{simple.lem}
Suppose that
$\nu \in \{\pm\}$
and
$(\al^\nu,\be^\nu) \in \B(\al_0^\nu,\be_0^\nu)$.
Then
\begin{equation}   \label{nu_Ga_zero_pos.eq}
\Ga^\nu(s,\th,\al^\nu,\be^\nu) = 0
  \implies
\nu  \, \Ga^\nu_s(s,\th,\al^\nu,\be^\nu) \,
\Ga^\nu_\th(s,\th,\al^\nu,\be^\nu) > 0.
\end{equation}
\end{lemma}

\begin{proof}

By a similar proof to that of Lemma~\ref{single_bc_simple.lem}
it can be shown that
\begin{equation}   \label{Ga_ab_zero_implies.eq}
\Ga^\nu(s,\th,\al^\nu,\be^\nu) = 0
\implies
\Ga^\nu_s(s,\th,\al^\nu,\be^\nu) \, \Ga^\nu_\th(s,\th,\al^\nu,\be^\nu)
\ne 0 .
\end{equation}
We now regard $(s,\th,\al^\nu,\be^\nu)$ as fixed,
and consider the equation
\begin{equation}  \label{Gz_eq_z.eq}
G(\tth,t) := \Ga^\nu(s,\tth,t \al^\nu,t \be^\nu) = 0,
\quad (\tth ,\ t) \in S^1 \X [0,1].
\end{equation}
It is clear that if $t \in [0,1]$ then
$(t\al^\nu,t\be^\nu) \in \B(\al_0^\nu,\be_0^\nu)$,
so by \eqref{Ga_ab_zero_implies.eq},
\begin{equation}  \label{Gz_th.eq}
G(\th,1) = 0 \quad \text{and} \quad
G(\tth,t)  = 0 \implies G_{\tth}(\tth,t)  \ne 0 .
\end{equation}
Hence, by \eqref{Gz_th.eq}, the implicit function theorem, and the
compactness of $S^1$,
there exists a $C^1$ solution function
$
t \to \tth(t) : [0,1] \to S^1 ,
$
for \eqref{Gz_eq_z.eq}
such that
$$
\tth(1) = \th, \quad
\Ga^\nu(s,\tth(t),t\al^\nu,t \be^\nu) = 0 , \quad
t \in [0,1]
$$
(the local existence of this solution function, near $t=1$, is trivial;
standard arguments show that its domain can be extended to include the
interval $[0,1]$ --- see the
proof of part~(b) of Lemma~\ref{s_bdd.lem} below for a similar
argument).

Next, by the definition of $\Ga^\nu$,
\eqref{nu_Ga_zero_pos.eq} holds
at $(s,\tth(0),0,0)$
and hence, by \eqref{Ga_ab_zero_implies.eq} and continuity,
\eqref{nu_Ga_zero_pos.eq} holds at
$(s,\tth(t),t \al^\nu,t \be^\nu)$ for all $t \in [0,1]$.
In particular, putting $t=1$ shows that \eqref{nu_Ga_zero_pos.eq} holds
at $(s,\th,\al^\nu,\be^\nu)$, which completes the proof
of Lemma~\ref{simple.lem}.
\end{proof}

We now return to the pair of equations \eqref{simul_zeros.eq}.
To solve these using the implicit function theorem we define the
Jacobian determinant
$$
J(s,\th,\bal,\bbe) :=
\begin{vmatrix}
\Ga^-_s(s,\th,\al^-,\be^-)  & \Ga^-_\th(s,\th,\al^-,\be^-)
\\[1 ex]
\Ga^+_s(s,\th,\al^+,\be^+)  & \Ga^+_\th(s,\th,\al^+,\be^+)
\end{vmatrix},
$$
for
$(s,\th,\bal,\bbe) \in (0,\infty) \X S^1 \X \B(\bal_0,\bbe_0).$
It follows from the sign properties of
$\Ga^\pm_s,\ \Ga^\pm_\th$
proved in Lemma~\ref{simple.lem} that
\begin{equation}  \label{Gz_Jnz.eq}
\Ga^+(s,\th,\al^+,\be^+) = \Ga^-(s,\th,\al^-,\be^-) = 0
\implies  J(s,\th,\bal,\bbe) \ne 0 ,
\end{equation}
and hence we can solve \eqref{simul_zeros.eq} for $(s,\th)$,
as functions of $(\bal,\bbe)$,
in a neighbourhood of an arbitrary solution of \eqref{simul_zeros.eq}.

Now suppose that
$(s,\th,\bal,\bbe) \in (0,\infty) \X S^1 \X
\B(\bal_0,\bbe_0)$
is an arbitrary (fixed) solution of \eqref{simul_zeros.eq}.
By \eqref{Gz_Jnz.eq} and the implicit function theorem
there exists a maximal open interval $\tI$ containing $1$ and a $C^1$
solution function
$$
t \to (\ts(t),\tth(t)) : \tI \to (0,\infty) \X S^1,
$$
such that
$$
(\ts(1),\tth(1)) = (s,\th),
\quad \Ga^\pm(\ts(t),\tth(t),t\al^\pm,t\be^\pm) = 0 ,
\quad t \in \tI .
$$
Furthermore, by Corollary~\ref{efun_in_P.cor}  and continuity,
there exists an integer $\tk \ge 0$ such that
\begin{equation}  \label{kz_defn.eq}
(\ts(t)^2,w(\ts(t),\tth(t))) \in P_{\tk}  ,
\quad t \in \tI .
\end{equation}

\begin{lemma}  \label{s_bdd.lem}
$(a)$\
There exists constants $C$,  $\de > 0$
such that
$\de \le \ts(t) \le C$, $t \in \tI;$
\\
$(b)$\
$0 \in \tI$.
\end{lemma}

\begin{proof}
(a)\ \
From the form of $w(s,\th)$, there exists
$C > 0$ such that
if $s \ge C$ then $(s^2,w(s,\th)) \not\in P_{\tk}$, for any $\th \in
S^1$.
Hence, by \eqref{kz_defn.eq}, $\ts(t) \le C$ for any $t \in \tI$.
Now suppose that the lower bound $\de>0$ does not exist,
so that we may choose a sequence $t_n \in \tI$, $n = 1,2,\dots,$ with
$\ts(t_n) \to 0.$
Writing $\ts_n := s(t_n)$, $\tth_n := \th(t_n)$ and
$\tw_n := w(\ts_n,\tth_n)$, $n = 1,2,\dots,$
it is clear that, as $n \to \infty$,
$$
\text{$|\tw'_n|_0 = {\rm O}(\ts_n)$
\quad  and  \quad
$|\tw_n - c_\infty|_0 \to 0$,}
$$
for some constant $c_\infty$
(after taking a subsequence if necessary,
and regarding $c_\infty$ as an element of $C^0[-1,1]$).
We now consider various cases.

Suppose that $c_\infty \ne 0$.
By \eqref{al_pm_strict_pos.eq},
$\al_0^\nu \ne 0$ for some $\nu \in \{\pm\}$,
and the corresponding boundary condition \eqref{slbc.eq} yields
$$
0 = \al_0^\nu \tw_n(\nu) - \sum^{m^\nu}_{i=1} \al^\nu_i
\tw_n(\eta^\nu_i)
 + {\rm O}(\ts_n)
\to c_\infty \Big(  \al_0^\nu - \sum^{m^\nu}_{i=1} \al^\nu_i \Big) ,
$$
which contradicts \eqref{AB_lin_cond.eq},
and so proves the existence of $\de > 0$  in this case.

Now suppose that $c_\infty = 0$.
Without loss of generality we also suppose that
$\tth_n \searrow 0$
(after taking a subsequence if necessary)
and so,
for all $n$ sufficiently large,
$|\tw_n|_0$ is attained at the end point $x = 1$.

Suppose that $\al_0^+ \ne 0$.
By the definition of $\tw_n$, we obtain from  \eqref{slbc.eq}
\begin{align*}
&  \ts_n
\Big(
\al_0^+ - \sum^{m^+}_{i=1} \al^+_i \eta^+_i
+
\be_0^+ - \sum^{m^+}_{i=1} \be^+_i
\Big)
 +
\tth_n
\Big(
\al_0^+ - \sum^{m^+}_{i=1} \al^+_i
\Big)
= {\rm O}(\ts_n^3 + \tth_n^3) ,
\end{align*}
but, by \eqref{albe_nz.eq}-\eqref{AB_lin_cond.eq}, the terms in the
brackets on the left hand side are strictly positive, so this is
contradictory when $n$ is sufficiently large.

Suppose that $\al_0^+ = 0$,
and so $\be_0^+ > 0$
(by \eqref{albe_nz.eq}, \eqref{albe_sign.eq}).
Dividing \eqref{slbc.eq} by $\ts_n$ and letting $n \to \infty$ yields
$$
0 = s_n^{-1} \Big(
\be_0^+ \tw_n'(1) - \sum^{m^+}_{i=1} \be^+_i \tw_n'(\eta^+_i) \Big)
\to \be_0^+ - \sum^{m^+}_{i=1} \be^+_i >0 ,
$$
by \eqref{AB_lin_cond.eq},
which is again contradictory.
This completes the proof of part (a) of Lemma~\ref{s_bdd.lem}.

\noindent
(b)\ \
Suppose that $0 \not\in \tI$,
and let $\hat t = \inf \{ t \in \tI\} \ge 0$.
By part (a) of the lemma, there exists a  sequence
$t_n \in \tI$, $n = 1,2,\dots,$
and a point $(\hat s,\hat \th) \in  (0,\infty) \X S^1$,
such that
$$
\lim_{n \to \infty} t_n = \hat t ,
\quad
\lim_{n \to \infty} (\ts(t_n),\tth(t_n)) = (\hat s,\hat \th) .
$$
Clearly, the  point
$(\hat s,\hat \th,\hat t\bal,\hat t\bbe)$
satisfies \eqref{simul_zeros.eq} so,
by the above results,
the solution function $(\ts,\tth)$ extends to an open neighbourhood of
$\hat t$,
which contradicts the choice of $\hat t$ and the maximality of the
interval $\tI$.
\end{proof}

For any given $(\bal,\bbe) \in \B(\bal_0,\bbe_0)$
the above arguments have shown that:
\begin{mylist}
\item
any solution
$(s,\th,\bal,\bbe) \in (0,\infty) \X S^1 \X  \B(\bal_0,\bbe_0)$
of \eqref{simul_zeros.eq} can be continuously connected to exactly one
of the solutions
$\{(s_k^\bzero,\th_k^\bzero,\bzero,\bzero) : k \ge 0\}$.
\end{mylist}
Similar arguments show that:
\begin{mylist}
\item[(b)]
any  solution
$\{(s_k^\bzero,\th_k^\bzero,\bzero,\bzero) : k \ge 0\}$
can be continuously connected to exactly one solution,
say
$(s_k(\bal,\bbe),\th_k(\bal,\bbe),\bal,\bbe) \in (0,\infty) \X S^1 \X
\B(\bal_0,\bbe_0)$,
of \eqref{simul_zeros.eq}.
\end{mylist}
Hence, for each $k \ge 0$,
we obtain the eigenvalue and eigenfunction
$$
(\la_k(\bal,\bbe),u_k(\bal,\bbe))  :=
(s_k(\bal,\bbe)^2, w(s_k(\bal,\bbe),\th_k(\bal,\bbe))) \in P_k ,
$$
and we see that there is no eigenvalue $\tla \ne \la_k(\bal,\bbe)$,
with eigenfunction $\tu$, for which $(\tla,\tu) \in P_k$.

Next, by Theorem~\ref{spec_zero.thm},
$s_k^\bzero = s_k^\bzero(\bzero,\bzero) <
s_{k+1}^\bzero = s_{k+1}^\bzero(\bzero,\bzero)$
and by Theorem~\ref{single_bc.thm},
$s_k(\bal,\bbe) \ne s_{k+1}(\bal,\bbe)$
for any $(\bal,\bbe) \in \B(\bal_0,\bbe_0)$,
so it follows from the continuation construction that
$s_k(\bal,\bbe) < s_{k+1}(\bal,\bbe)$
for all $(\bal,\bbe) \in \B(\bal_0,\bbe_0)$.

Finally, for fixed $(\bal,\bbe)$,
the fact that
$(\la_k(\bal,\bbe),u_k(\bal,\bbe)) \in P_k$, for $k \ge 1$,
shows that as $k \to \infty$ the oscillation count tends to $\infty$,
so by standard properties of the differential equation
\eqref{eval_de.eq}
we must have
$\lim_{k\to\infty} \la_k = \infty$.
This concludes the proof of Theorem~\ref{spec.thm}.
\end{proof}

The implicit function theorem construction of $\la_k$ and $u_k$ in the
proof of Theorem~\ref{spec.thm} also imply continuity properties
which will be useful below, so we state these in the following corollary
(continuity of $u_k$ will be in the space $C^0[-1,1]$, although stronger
results could easily be obtained).

\begin{cor}  \label{cts_evals.cor}
For each $k \ge 0$, $\la_k \in \R$ and $u_k \in C^0[-1,1]$ depend
continuously on
$(\bal_0,\bbe_0,\bal,\bbe,\bfeta) \in
\B \X (-1,1]^{m^-} \X [-1,1)^{m^+}$.
\end{cor}

\subsection{Positivity of the principal eigenfunction}
\label{positivity.sec}

In many applications it is important to know that the
{\em principal eigenfunction} $u_0$ is positive.
Thus we will now consider conditions which ensure this is true.

\begin{thm}  \label{posrinc_efun.thm}
Suppose that
\eqref{albe_nz.eq}-\eqref{AB_lin_cond.eq}
hold,
and $\al^\pm \ge 0$.
Then$:$\\
$(a)$\ $u_0 > 0$ on $(-1,1);$\\
$(b)$ if $\be_0^\nu \ne 0$, for some $\nu \in \{\pm\}$,
then $u_0(\nu) > 0$.
\end{thm}

\begin{proof}
By standard Sturm-Liouville theory the result is true when
$(\bal,\bbe) = (\bzero,\bzero)$
(part (a) is standard and (b) follows immediately since, under the
stated hypotheses,  $u_0(\nu) = 0 \Rightarrow u_0'(\nu) = 0$,
and an eigenfunction cannot have a double zero).
Now suppose that both $\be_0^\pm \ne 0$.
If the result fails then,
by using a limiting argument in the construction of the
eigenvalues by continuation from
$(\bal,\bbe) = (\bzero,\bzero)$
in the proof of Theorem~\ref{spec.thm},
we can show that there exists some
$(\bal,\bbe) \in \B(\bal_0,\bbe_0)$, with $\al^\pm \ge 0$,
such that the principal eigenfunction
$u_0(\bal,\bbe) \ge 0$ satisfies:
\begin{mylist}
\item[(1)]
$u_0(\bal,\bbe) > 0$ on $(-1,1)$\ \
(since $u_0(\bal,\bbe)$ cannot have a double zero);
\item[(2)]
$u_0(\bal,\bbe)(\nu) = 0$,
and hence
$|u_0'(\bal,\bbe)(\nu)| = |u_0'(\bal,\bbe)|_0$,
for some $\nu \in \{\pm\}$.
\end{mylist}
Now, by \eqref{slbc.eq}-\eqref{AB_lin_cond.eq}
\begin{align*}
0 &= \be_0^\nu u_0'(\bal,\bbe)(\nu ) -
\sum_{i=1}^{m^\nu}
\al_i^\nu u_0(\bal,\bbe)(\eta_i^\nu)
-
\sum_{i=1}^{m^\nu}
\be_i^\nu u_0'(\bal,\bbe) (\eta_i^\nu)
\\[1 ex]
& \le
 - |u_0'(\bal,\bbe)|_0
 \Big( |\be_0^\nu| - \sum_{i=1}^{m^\nu} |\be_i^\nu| \Big)
 < 0 ,
\end{align*}
and this contradiction shows that this case cannot occur.

Next, suppose that one, or both, of $\be_0^\pm=0$.
We replace the coefficients
$\be_0^\pm$ by $\be_0^\pm \pm 1/n$, $n = 1,2,\dots,$
and then let  $n \to \infty$.
By the result just proved, each of the corresponding principal
eigenfunctions, say $u_{0,n} \ge 0$, have the properties (a) and (b),
and so by Corollary~\ref{cts_evals.cor}
the limiting eigenfunction, say $u_{0,\infty} \ge 0$, satisfies (a),
and we can now prove that $u_{0,\infty}$ satisfies (b) by the same
calculation as before.
\end{proof}

\subsection {Algebraic multiplicity}
\label{alg_mult.sec}

Throughout this section we will suppose that
\eqref{al_pm_strict_pos.eq} holds so that,
by Theorem~\ref{De_inverse.thm}, $\De$ has an inverse
operator $\De^{-1} : Y \to X$
(see Remark~\ref{neu_comp_res.rem} below for some comments on the
Neumann-type case, when \eqref{al_pm_strict_pos.eq} does not hold).
We can also regard this inverse as an operator
$\De^{-1} : Y \to Y$,
which we will denote as $\De^{-1}_Y$.
Since $X$ is compactly embedded into $Y$,
$\De^{-1}_Y$ is compact
(indeed, this compactness together with the fact that $\De^{-1}_Y$ maps
$Y$ into itself is the motivation for introducing $\De^{-1}_Y$).
Now, the eigenvalue problem \eqref{mp_eval.eq} is equivalent to
the equation
\begin{equation}  \label{K_la_u.eq}
(I_Y + \la\De^{-1}_Y) u = 0,  \quad u \in Y ,
\end{equation}
where $I_Y$ denotes the identity on $Y$.
Hence, each eigenvalue $\la_k$, $k = 0,1,\dots,$
can be regarded as a characteristic value of $- \De^{-1}_Y$.
As usual, we  define the algebraic multiplicity of the
characteristic value $\la_k$ to be
$$\dim \bigcup_{j=1}^\infty N((I_Y + \la_k \De^{-1})^j)$$
(where $N$ denotes null-space).

\begin{lemma}  \label{alge_mult.lem}
For each $k \ge 0$ the algebraic multiplicity of the characteristic
value $\la_k$  of $-\De^{-1}_Y$ is equal to 1.
\end{lemma}

\begin{proof}
The proof is again by continuation with respect to $(\bal,\bbe)$,
so we now write $\De^{-1}_Y(\bal,\bbe)$ and $\la_k(\bal,\bbe)$,
for $(\bal,\bbe) \in \B(\bal_0,\bbe_0)$.
When $(\bal,\bbe)= (0,0)$ it is easy to see
that the algebraic multiplicity of $\la_k(0,0)$ is equal to 1
(this case corresponds to the standard Sturm-Liouville problem).
Next, it was shown in
Corollaries~\ref{De_inverse_cts.cor)} and~\ref{cts_evals.cor} that
$\De^{-1}_Y(\bal,\bbe)$  and $\la_k(\bal,\bbe)$
depend continuously on $(\bal,\bbe)$,
and Theorem~\ref{spec.thm} shows that as $(\bal,\bbe)$ varies over
$\B(\bal_0,\bbe_0)$,
eigenvalues with different $k$ never meet.
Hence, by the results in \cite[Ch. 2, Sec. 5]{KAT},
the algebraic multiplicity of $\la_k(\bal,\bbe)$
is constant for $(\bal,\bbe) \in \B(\bal_0,\bbe_0)$
(the discussion in \cite[Ch. 2, Sec. 5]{KAT} is in finite dimensions
but,
as noted there, the results extend to bounded operators in infinite
dimensions).
This proves the result.
\end{proof}

\begin{remark}
In the case $\al^- = 0$, $\al^+>0$, $\be^\pm = 0$,
Lemma~\ref{alge_mult.lem} was proved directly in
\cite[Lemma~2.6]{XU} and \cite[Lemma~3.8]{RYN3}
(that is, without relying on perturbation theory for linear operators),
but it seems to be difficult to extend this proof to the general case.
This result was extended to
general Dirichlet-type and Neumann-type problems in
\cite{RYN5} and \cite{RYN6} respectively.
\end{remark}

\begin{remark}  \label{neu_comp_res.rem}
For simplicity we have excluded the Neumann-type case from this
section, since in this case the operator $\De$ does not have an inverse.
Of course, one could consider the operator $\De - \mu I_Y$, with
$\mu > 0$; it can be shown that this operator has an inverse, which is
compact (as a mapping into $Y$),
that is, $\De$ has compact resolvent.
We could then obtain similar results to those above.
However, this would entail considerable additional notational
complexity, and the Neumann-type case was treated in detail in
\cite{RYN6}, so we will simply omit this case here.
\end{remark}

\subsection{Counter examples}  \label{eval_counterexamples.sec}

In this section we will show that Theorem~\ref{spec.thm}
need not be true if \eqref{AB_lin_cond.eq} does not hold, and that the
condition \eqref{AB_lin_cond.eq} is, in some sense, optimal for the
validity of Theorem~\ref{spec.thm}.
In fact, for the Dirichlet-type problem,
it was shown in \cite[Examples 3.5, 3.6]{RYN3} that if
$\sum_{i=1}^{m^\pm} |\al_i^\pm| = \al_0^\pm$
then we may have an eigenvalue/eigenfunction pair
$(\la,u) \in \pa P_k$, for some $k$
(in the present notation)
while if
$\sum_{i=1}^{m^\pm} |\al_i^\pm| > \al_0^\pm$
then we may have $\si_k = \emptyset$ for a finite, but arbitrarily
large, set of integers $k$,
that is, the corresponding eigenvalues $\la_k$ may be `missing' from the
sequence of eigenvalues constructed in Theorem~\ref{spec.thm}.
Similar examples were constructed for the Neumann-type case in
\cite[Examples 4.17, 4.18]{RYN6}.
These examples show that condition \eqref{AB_lin_cond.eq} is optimal
in the cases where one or other of the fractions on
the left hand side of \eqref{AB_lin_cond.eq} is absent.
Thus it seems of interest to also show that \eqref{AB_lin_cond.eq} is
optimal when both fractions are present.
The following example will do this when these fractions are nonzero
and equal to each other.
More precisely, in this case we will show that if the number $1$ on
the right hand side of \eqref{AB_lin_cond.eq} is increased
by an arbitrarily small amount then Theorem~\ref{spec.thm} need not
hold, and arbitrarily many eigenvalues may be `missing'.

For notational simplicity we will consider the problem on the interval
$(0,1)$, with a standard Dirichlet condition at $x=0$, and the following
multi-point condition at $x=1$
\begin{equation}  \label{c_ex_bc.eq}
\al_0 u(1) + \be_0 u'(1)
=
\al_1 u(\eta_1) + \be_2 u'(\eta_2) .
\end{equation}
For any eigenvalue $\la = s^2 > 0$ the
corresponding eigenfunction must have the form $C \sin s x$, $C \in \R$.
Hence, defining $\Ga : \R \to \R$ by
$$
\Ga(s) :=
\al_0 \sin s + s \be_0 \cos s
 -
\al_1 \sin s\eta_1 + s \be_2 \cos s \eta_2 ,
\quad s \in \R ,
$$
it is clear that $\la = s^2$ is an eigenvalue iff $\Ga(s) = 0$,
and also, for any integer $k \ge 0$,
$\la \in \si_k \implies \la \in [(k-2)\pi,(k+2)\pi]$.

To construct our counter example we will show that with a suitable
choice of the coefficients in the  boundary condition
\eqref{c_ex_bc.eq} there exists a `long' interval  $I$ such that if $s
\in I$ then $\Ga(s) \ne 0$,
that is, $s^2$ cannot be an eigenvalue.
This will show that $\si_k = \emptyset$ for a range of values of $k$.

Choose a `large' integer $k_0$
(we will be more specific below), and set:
$$
\ep = \frac{10}{k_0}, \quad  s(\ga) = (1+\ga\ep) k_0 \pi,
\quad \ga \in [-1,1] .
$$
Hence, as $\ga$ varies over the interval $[-1,1]$,
the number $s(\ga)$ varies over the interval
$$I_{k_0} := [(k_0-10)\pi,(k_0+10)\pi]. $$
We also set:
\begin{alignat*}{2}
\al_0 &= 1,  &  \be_0 &= \frac{1}{k_0\pi},
\\[1 ex]
\al_1 &= \frac{1+\ep}{\sqrt{2}} ,
\qquad
&  \be_2 &= \frac{1}{k_0\pi} \frac{1+\ep}{\sqrt{2}} ,
\\[1 ex]
\eta_1  &= \frac{1}{2k_0} , &  \eta_2 &=  \frac{1}{k_0}.
\end{alignat*}
Simple estimates now show that if $\ep$ is sufficiently small
(that is, if $k_0$ is sufficiently large)
then, for $\ga \in[-1,1]$,
\begin{align*}
\Ga(s(\ga)) & \le
\sqrt{2} + \ep - \frac{1+\ep}{\sqrt{2}}
\Big( \sin \frac{\pi}{2}(1+\ep) - (1+\ep) \cos \pi(1+\ep)  \Big)
\\
& \le
\sqrt{2} + \ep - \sqrt{2}(1+\ep)(1-\ep/14)
\\
& <
\ep \big( 1 - \frac{13\sqrt{2}}{14} + {\rm O}(\ep) \big)
\\
& <
0 .
\end{align*}
This shows that there is no eigenvalue $\la = s^2$ with
$s \in I_{k_0}$,
that is, $\si_{k} = \emptyset$ if $k \in [k_0 - 7,k_0 + 7]$.
Clearly, there is nothing special about the number 10 in this example,
so in fact we could construct an example for which $\si_k = \emptyset$
for an arbitrarily long succession of integers $k$.
Also, since
$$
\frac{\al_1}{\al_0}  = \frac{\be_2}{\be_0}  = \frac{1+\ep}{\sqrt{2}} ,
$$
and $\ep$ is arbitrarily small,
we see that if the number $1$ in condition
\eqref{AB_lin_cond.eq}
is increased by an arbitrarily small amount then Theorem~\ref{spec.thm}
need not hold.

\end{document}